\documentclass[twoside,11pt]{article}

\usepackage{ifpdf}
\usepackage{jmlr2e}
\usepackage{framed}
\usepackage{subfig}
\usepackage{algorithmic}
\usepackage{algorithm}
\usepackage{listings}
\usepackage{xcolor}
\usepackage{paralist}
\usepackage{graphicx}
\usepackage{ifpdf}
\usepackage{amsmath}
\usepackage{color}
\usepackage{float}
\usepackage{textcomp}
\usepackage{amssymb}
\usepackage{hyperref}
\usepackage{multirow}
\usepackage{endnotes}
\usepackage{sidecap}

\ShortHeadings{Crossing Minimization within Graph Embeddings}{Shabbeer, Ozcaglar and Bennett}
\firstpageno{1}

\begin{document}

\title{Crossing Minimization within Graph Embeddings}

\author{\name Amina Shabbeer $\dagger$ \email shabba@cs.rpi.edu      \AND
       \name Cagri Ozcaglar $\dagger$ \email ozcagc2@cs.rpi.edu        \AND
       \name Kristin P. Bennett $\dagger$* \email bennek@rpi.edu\\
       \addr  Department of Computer Science $\dagger$\\
       Department of Mathematical Sciences*\\
       Rensselaer Polytechnic Institute\\
       Troy, NY-12180, USA}

\editor{}

\maketitle

\begin{abstract}
We propose a novel optimization-based approach to embedding heterogeneous high-dimensi\-onal data characterized by a graph. The goal is to create a two-dimensional visualization of the graph structure such that edge-crossings are minimized while preserving proximity relations between nodes. This paper provides a fundamentally new approach for addressing the crossing minimization criteria that exploits Farkas' Lemma to re-express the condition for no edge-crossings as a system of nonlinear inequality constraints. The approach has an intuitive geometric interpretation closely related to support vector machine classification. While the crossing minimization formulation can be utilized in conjunction with any optimization-based embedding objective, here we demonstrate the approach on multidimensional scaling by modifying the stress majorization algorithm to include penalties for edge crossings. The proposed method is used to (1) solve a visualization problem in tuberculosis molecular epidemiology and (2) generate embeddings for a suite of randomly generated graphs designed to challenge the algorithm. Experimental results demonstrate the efficacy of the approach. The proposed edge-crossing constraints and penalty algorithm can be readily adapted to other supervised and unsupervised optimization-based embedding or dimensionality reduction methods. The constraints can be generalized to remove overlaps between any graph components represented as convex polyhedrons including node-edge and node-node intersections.\end{abstract}

\section{Introduction}

A good graph visualization clearly and effectively describes the nodes of a graph as well as the underlying relationships between these nodes. In this work, we propose a novel approach to embedding a high-dimensional graph in a two-dimensional space such that crossings are minimized while proximity relations between nodes are simultaneously preserved. The quality of a visualization is gauged on the basis of how easily it can be understood and interpreted. In this context, a minimum number of edge crossings has been identified as the most desirable characteristic for graph visualizations \citep{purchase1997aesthetic,ware2002cognitive,battista}. Minimizing crossings is a challenging problem. Determining the minimum number of crossings for a graph is NP-complete \citep{garey1983crossing}. In practice as well, this is a very difficult problem. Existing state of the art integer programming based \emph{exact} crossing minimization approaches for general graphs are only tractable for small sparse graphs. Graphs of upto 100 nodes can be solved in 30 minutes of computation time, and reaching the optimal solution is not guaranteed\citep{exactcrossing,oocm}. 
There exists a body of literature on the problems of drawing restricted classes of graphs e.g. 2-layer or k-layer crossing minimized graphs \citep{2layerplanar, hierarchplanar}, which are also NP-complete problems and difficult in practice. Polynomial time embedding algorithms do exist for the special class of planar graphs. However, these classes of methods do not allow for much flexibility in placement of nodes implying limited additional constraint satisfaction capability such as for preserving proximity relations. Additionally, the topological transformations involved alter the user's mental map of the data that may be based on local structure or relative proximities. 

However, proximity preservation is a much desired property for general data embedding \citep{sne,lle}. Frequently, the nodes of a graph represent objects that have their own intrinsic properties with associated distances or similarity measures that describe implicit relations between all pairs of nodes. A graph embedding  that serves to represent such relationships faithfully must produce a mapping of nodes from high-dimensional space to low-dimensional vectors that preserves pairwise proximity relations. For general data embedding, the desired quality is frequently expressed as a function of the embedding and then optimized. For example, in Multidimensional Scaling (MDS), the goal is to produce an embedding that minimizes the difference between the actual distances and (Euclidean) distances in the embedding between all pairs of nodes. Several nonlinear dimensionality reduction techniques have been proposed where the emphasis is on preserving local structure \citep{laplacianeigenmaps,sne,tsne}. While these methods can be applied to visualize the nodes of the graph, the extant edges between nodes create a large number of edge-crossings. 

Thus, embedding such heterogeneous data poses a unique challenge. The underlying graph structure must be presented clearly. At the same time, since the nodes of the graph are themselves data points characterized by features, the positions of nodes in the embedding must effectively represent the proximities between the data points in the original space. A natural question is: how can the number of edge crossings in the embedded graph be minimized, while simultaneously optimizing the embedding objective? This requires expressing the basic condition for no edge crossings as an optimization problem, which has previously not been done \footnote{Existing integer programming based crossing-number formulations incorporate a large number of constraints e.g. \emph{Kuratowski} constraints characterizing planar subgraphs based on the theorem that a graph is planar iff it contains no Kuratowski subdivisions \citep{berge1958theorie}.}. The formulation of edge-crossings constraints and the representation of crossing minimization as a continuous optimization problem are the principle contributions of this paper. This representation is a fundamentally new paradigm for crossing minimization. Expressing edge-crossing minimization as a \emph{continuous} optimization problem offers the additional advantage that other embedding objectives such as proximity preserving criteria can be simultaneously optimized. 
  
The key theoretical insight of the paper is that the condition that two edges do {\bf not} cross is equivalent to the feasibility of a system of nonlinear inequalities.  In Section \ref{sec:cntnsedgecrossing}, we prove this using a theorem of the alternative: Farkas' Lemma. The transformed system  has an intuitive geometric interpretation, it ensures that the two edges are separated by a linear hyperplane.  Thus, each edge-crossing constraint reduces to a classification problem which is very closely related to support vector machines (SVM). The resulting system of inequalities can then be relaxed to create a natural penalty function for each possible edge crossing. This non-negative function goes to zero if no edge crossings occur. This general approach is applicable to the intersection of any component of the graph represented as a convex polyhedron, including arbitrary shaped nodes, labels and subgraphs represented by their convex hulls. The approach is a distinct and important departure from prior crossing minimization approaches \citep{orth,exactcrossing,gutwenger2004experimental} and graph drawing algorithms that employ heuristics to avoid crossings \citep{fdpfruchterman,wills1999nicheworks}. It also allows for extensions of general data embedding methods applied to graphs that otherwise ignore crossings.

In Section \ref{sec:CR-SM}, we explore how edge-crossing constraints can be added to stress majorization algorithms for MDS.  
We develop the Crossing Reduction with Stress Majorization (CR-SM) algorithm which simultaneously minimizes stress while eliminating or reducing edge crossings using penalized stress majorization. The method solves a series of unconstrained nonlinear programs.  We test the method on two sets of graph embedding problems with associated distance matrices. We first demonstrate the approach on a compelling problem involving genetic distances in tuberculosis molecular epidemiology.   The graphical results are shown for spoligoforests drawn using a set of fifty-five biomarkers.  The method found planar graph embeddings with lower stress than those generated using the state-of-the-art Graphviz NEATO algorithm (stress majorization for MDS).  We then demonstrate the approach on randomly generated high-dimensional graphs designed to have some planar embeddings with high stress. The results show that the proposed approach, CR-SM, can produce two-dimensional embeddings with minimal edge crossings with little increase in stress.  Additional illustrations are provided in the appendix. Animations of the algorithm illustrating how the edge crossing penalty progressively transforms the graphs are provided at \url{http://www.cs.rpi.edu/~shabba/FinalGD/}.

We now describe our notation.  All vectors will be column vectors unless transposed to a row vector by a prime $'$. For a vector $x$ in the \emph{n-}dimensional real space $\mathbb{R}^n$, $x_i$ represents the $ith$ component. $x_+$ denotes the vector in $\mathbb{R}^n$  with components $(x_+)_i=max(x_i, 0), i=1,..,n$. The 1-norm of $x$,  $\displaystyle\sum\limits_{i=1}^n |x_i|$ will be denoted by $||x||_1$, and the 2-norm of $x$, $\sqrt{x'x}$ will be denoted by $||x||_2$. A vector of ones in a real space of arbitrary dimension will be denoted by $e$. The notation $A \in \mathbb{R}^{m \times n}$ will signify a real  $m \times n$ matrix. For such a matrix, $A'$ will denote the transpose, while $A_i$ will denote the $ith$ row. The notation $\displaystyle{\underset{x \in X}{argmin f(x)}}$ represents the solution to the problem $\displaystyle{\underset{x \in X}\min f(x)}$ and equals $\{x^* \in X: f(x^*) \leq f(x), \forall x \in X\}$. 

\section{Motivation}
\label{sec:Motivation}
\begin{figure}
\includegraphics[scale=0.5]{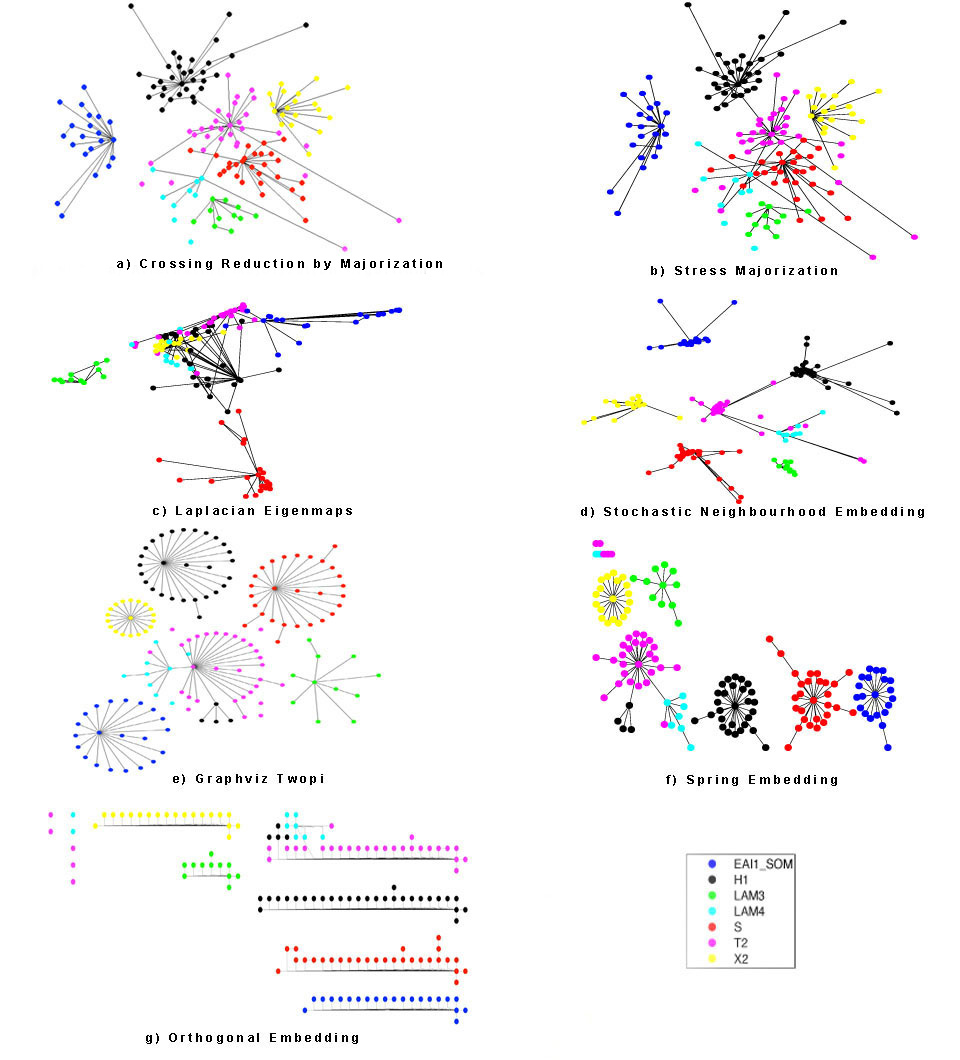}
\caption{Embeddings of spoligoforests of SpolDB4 sublineages by 7 algorithms 
(a) CR-SM
(b) MDS
(c) Laplacian Eigenmaps
(d) Stochastic Neighborhood Embedding (SNE)
(e) Graphviz Twopi
(f) Spring Embedding
(g) Orthogonal Embedding.
The proposed approach CR-SM shown in (a) eliminates all edge crossings with little change in the overall stress as compared with the stress majorization solution in (b).
}
\label{fig:SpolDB4}
\end{figure}
The goal is to create clear representations of graphs with a reduced number of crossings. This must be achieved while optimizing other embedding criteria, such as preserving pairwise distances between nodes as defined in the original high-dimensional space. Such graphs with associated pairwise distances between nodes arise in various domains; the specific motivating application for this work is visualization of phylogenetic forests (spoligoforests) of  \emph{Mycobacterium tuberculosis} complex (MTBC).  In a spoligoforest, each node represents a strain of MTBC described by its genetic fingerprint and each edge represents a putative mutation \citep{webtools}.  Each node or strain has a genetic distance that can be defined to every other strain, even if they are not connected in the underlying graph (phylogenetic forest). 
Similar graphs exist in other domains, e.g. in a graph of web pages,  each node may be a web page and each edge may represent a hyperlink between pages.  Each web page is a document with intrinsic properties, so there is an associated distance or similarity measure between nodes even if no link exists between them.

Existing embedding and graph drawing methods typically do a good job at preserving proximity relations or minimizing crossings, but not both.
 Here we illustrate the shortcomings of existing methods. Figure \ref{fig:SpolDB4} shows the visualization of a (planar) spoligoforest for the SpolDB4 subfamilies of MTBC created by the proposed approach and 6 other embedding methods:  \begin{inparaenum}[\itshape \upshape(a\upshape)]
\item CR-SM
\item MDS
\item Laplacian Eigenmaps
\item Stochastic Neighborhood Embedding
\item Graphviz Twopi
\item Spring Embedding
\item Orthogonal Embedding.
\end{inparaenum}
The proposed approach, CR-SM, minimizes crossings while preserving proximity relations. MDS by stress majorization as implemented in Graphviz NEATO \citep{majorization} shown in Figure \ref{fig:SpolDB4}(b) preserves pairwise distances specified between all pairs of nodes but has edge crossings. The Laplacian Eigenmap technique involves inferring an adjacency matrix and the corresponding weighted Laplacian based on distances between points in high-dimensional space and subsequently generating a spectral embedding based on the weighted Laplacian \citep{laplacianeigenmaps}. Limitations of Laplacian Eigenmaps reported in \citet{van2007dimensionality} are observed in the spoligoforest visualizations. Multiple nodes collapse to form dense clusters of points in the reduced space leading to collocated edges. This makes it difficult to observe individual nodes and relations between them. The graph in Figure \ref{fig:SpolDB4}(d) is generated using SNE which is based on representing proximities in high-dimensional space as conditional probabilities and generating embeddings in the reduced space that preserve these probabilities \citep{sne}. While genetically similar strains cluster together in the embedding generated by SNE, there are edge crossings and the genetic relatedness between all pairs of strains is less evident as indicated by high stress tabulated in Table 1, Section \ref{sec:results}. Embeddings generated using Graphviz Twopi that tries to preserve uniform angular span, have a visually appealing radial layout but do not represent pairwise distances between nodes. The spring embedding method treats a graph as a physical system, in which all nodes exert repulsive forces on each other, while nodes connected by edges also have attractive forces between them \citep{fdpfruchterman}. In the state of equilibrium, nodes connected by edges are placed close to each other; as the edges are short the number of crossings is low. A variation of the spring model for drawing weighted graphs defined by \citet{kamada} is based on the assumption that edge lengths need to be preserved. It does not represent the majority of pairwise distances between all pairs of nodes not directly connected by edges. Orthogonal layout methods use heuristics to generate straight-line planar embeddings with minimum edge bends \citep{orth} as shown in Figure \ref{fig:SpolDB4}(g). This method first generates a ``visibility representation'', which is a skeletal representation of the graph. Individual graph components are then substituted with equivalent straight-line forms. Further heuristics-based transformations are made to generate an orthogonal embedding with minimum edge-bends. While the planar embeddings in Figure \ref{fig:SpolDB4}(e)-(g) may be uncluttered and therefore visually appealing, they are inaccurate representations of the data.  The relative placement of individual components of the graph do not represent genetic distances defined by the distance matrix. For example, the relative placement of subgraphs representing MTBC lineages does not reflect the genetic similarity between lineages and edge lengths do not represent the extent of mutation. The proposed approach in Figure \ref{fig:SpolDB4}(a) on the other hand, represents pairwise distances correctly. By optimizing the MDS objective or any other dimensionality reduction objective with additional edge crossing penalties, the embedding has no edge-crossings and also a naturally emerging radial structure due to the minimum separation between edges enforced.

\section{Continuous Edge-Crossing Constraints}
\label{sec:cntnsedgecrossing}

We show how edge-crossing constraints can be expressed as a system of nonlinear inequalities through the introduction of additional variables for each edge crossing. Expressing the constraint that two edges must not cross as a system of nonlinear equalities is a key non-obvious first step for developing a continuous objective function to minimize edge crossings. The formulation is based on the fact that each straight-line edge is a convex polyhedron. Therefore, $\exists u \not=0$ and $\gamma$ such that for two edges that do not intersect, $x'u-\gamma$ is nonpositive for all points $x \in \mathbb{R}^2$ lying on one edge and nonnegative for all points $x \in \mathbb{R}^2$ lying on the other edge. Figure \ref{fig:cross} illustrates that the no-edge-crossing constraint corresponds to introducing a separating hyperplane defined by $\{x|x \in \mathbb{R}^2, x'u=\gamma\}$ and requiring each edge to lie in opposite half spaces. 
\begin{figure} [H]
 \begin{tabular}{cc}
(a) \;\;\; \includegraphics{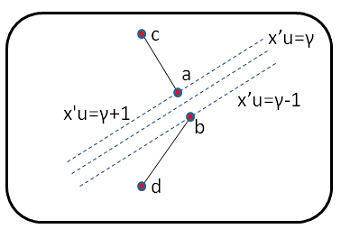}
     (b) \;\;\; \includegraphics{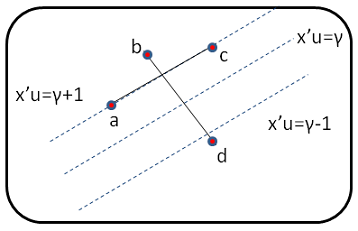}
 \end{tabular}
  \caption{ In (a) edge {$\cal{A}$}  from $a$ to $c$ and edge  $\cal{B}$ from  $b$ to $d$ do not cross.   Any line between  {$x'u-\gamma=1$} and {$x'u-\gamma=-1$} strictly separates the edges.  Using a soft margin, the plane in (b) $x'u-\gamma=0$ separates the plane into half spaces that should contain each edge.
}
 \label{fig:cross}
\end{figure}
To elucidate this idea further, note that each point on an edge can be represented as the convex combination of the extreme points of the edge.  Consider edge $\cal A$ with end points $a=[a_x \;\; a_y] \in \mathbb{R}^2$ and $c=[c_x \;\; c_y] \in \mathbb{R}^2$ and edge $\cal B$  with end points $b=[b_x \;\; b_y] \in \mathbb{R}^2$ and $d=[d_x \;\; d_y] \in \mathbb{R}^2$. The matrices $A$ and $B$ contain the end or extreme points of the edges $\cal A$ and $\cal B$  respectively.  Any point in the intersection of edge $\cal A$ and $\cal B$ can be written as a convex combination of the extreme points of $\cal A$ and also as a convex combination of the extreme points of $\cal B$. Therefore, two edges do not intersect if and only if the following system of equations has {\bf no solution}:
\begin{align}
\exists \;\;\; \delta_{A} \in \mathbb{R}^2\mbox{ and } \delta_{B}  \in \mathbb{R}^2 \mbox{ such that}\;\;\;
A' \delta_{A}=B' \delta_{B} \;\;\;
e' \delta_{A}=1 \;\;\;
e' \delta_{B}=1 \;\;\;
\delta_{A} \geq 0 \;\;\;
\delta_{B} \geq 0
\end{align}
where
 $e$ is a 2-dimensional vector of ones and $\displaystyle{A= \begin {bmatrix} a_x & a_y \\c_x & c_y \end{bmatrix}}$ and $\displaystyle{B= \begin {bmatrix} b_x & b_y \\d_x & d_y \end{bmatrix}}$.
 
The conditions that two given edges do \emph{not} cross, i.e. that (1) has no solution, are precisely characterized by using Farkas' Lemma. Farkas' Lemma states that for each fixed $p \times n$ matrix $D$  the linear system $Du\le0$, $d'u>0$ has no solution $u \in  \mathbb{R}^n$ if and only if   the system  $D'v=d$, $v\ge 0$ has a solution $v \in  \mathbb{R}^p$.

\begin{theorem}[Conditions for no edge crossing] The edges $\cal A$ and $\cal B$ do not cross if and only if $\exists$ $u$, $\alpha$, $\beta \in \mathbb{R}^2$, such that 
\begin{equation}
Au \ge \alpha e \;\;\;\;\;
Bu \le \beta e \;\;\;\;\;
\alpha - \beta >0.\\
\label{eq:cond}
\end{equation}
\label{th:nocross}
\end{theorem}
\begin{proof}
By Farkas' Lemma, (\ref{eq:cond}) has a solution if and only if the following system has no solution
\begin{equation}
\label{eq:cond2a}
A'\delta_A -B'\delta_B =0, \;\;\;\;e'\delta_A=1 \;\;\;\; e'\delta_B=1 \;\;\;\;\delta_A \ge 0 \;\;\;\;\;\; \delta_B \ge 0.\\
\end{equation}
System \ref{eq:cond2a} has no solution if and only if the convex combination of the extreme points of $\cal A$ and $\cal B$ do not intersect.
\end{proof}

The theorem can be generalized to intersections between any component of the graph represented as convex combinations of their respective extreme points. Let the graph component represented as convex polyhedron ${\cal A}$ have $v$ extreme points and the component ${\cal B}$ have $s$ extreme points. There exists 
$\delta_A \in \mathbb{R}^v$ and $\delta_B \in \mathbb{R}^s$ such that ${\cal A}=\{ x |x = A' \delta_A, e'\delta_A=1, \delta_A\ge 0\}$ and ${\cal B}=\{ x |x = B'\delta_B, e'\delta_B=1, \delta_B\ge 0\}$. Then from Theorem \ref{th:nocross},

\begin{corollary}[Conditions for no intersection of two polyhedrons]
The polyhedrons ${\cal A}$ and ${\cal B}$ do not intersect, ${\cal A}\cap{\cal B}=\emptyset$, if and only if $\exists$ $u \in \mathbb{R}^2$, $\gamma \in \mathbb{R}$, such that 
\begin{equation}
\label{eq:cond2}
\begin{array}{l}
Au - \gamma e \ge  e\mbox{,}\;\;\;\;\;
Bu - \gamma e \le -e \\
\end{array}
\end{equation}
\end{corollary}
Therefore, two edges (or more generally two polyhedrons) {\bf do not} intersect if and only if
\begin{equation}
\label{eq:LP}
0=\min_{u,\gamma} f(A,B, u, \gamma)=\min_{u,\gamma} || (-Au+(\gamma +1)e)_+||^q_q + ||(Bu-(\gamma -1)e)_+||^q_q 
\mbox { for } q=1 \mbox { or }\ q=2. \end{equation}
Here $q=1$ represents the one-norm, while $q=2$ represents the least-squares form. As in SVMs, (\ref{eq:LP}) can be converted into a linear or quadratic program depending on the choice of $q=1$, or $q=2$ respectively \citep{vapnik,bradley1998feature}. Here we study $q=2$. Thus, the edge-crossing condition can be converted to alternate forms and potentially more convenient mathematical programs by choice of appropriate norms and introduction of constraints and extra variables to eliminate the plus function. The minimization function in (\ref{eq:LP}) provides a natural function for penalizing edges that do cross. Much like soft-margin SVM classification \citep{cortes1995support}, two edges (or more generally two polyhedrons) $\cal A$ and $\cal B$ do not intersect if and only if there exists a hyperplane, $x'u-\gamma=0$ that strictly separates the extreme points of $\cal A$ and $\cal B$. If the edges do not cross, then the optimal objective of (\ref{eq:LP}) will be 0; while it will be strictly greater than 0 if the edges do cross.

\section{Crossing Reduction with Stress Majorization}
\label{sec:CR-SM}

Edge-crossing constraints and penalties are a practical and flexible paradigm that can be used directly or as part of optimization-based graph embedding algorithms to minimize edge crossings in graph embeddings. Many algorithms are possible depending on the variant of the formulation of the penalty terms used. In this paper, we used the differentiable least-squares loss for the penalty terms. 
For  $\ell$ edge objects there are $\frac { \ell(\ell-1)} {2} $ possible intersections and thus $O(\ell^2)$ penalty terms.  Note, an optimized embedding will typically only produce a fraction of the possible edge crossings.  Thus, this suggests an iterative penalty approach would be efficient because we need only deal with the small set of edge crossings that actually occur during the iterations of the algorithm. In this section, we describe an algorithm for minimizing the \emph{stress} (MDS objective) with quadratic penalties for edge crossings. 
 
Note that alternately, by including 1-norm penalties for every edge-crossing, an exact penalty method can be developed for crossing minimization.  Such \emph{exact} penalty functions have the desirable property that a single minimization can yield the \emph{exact} solution \citep{nocedal}. While this offers advantages on convergence rates as a finite penalty would suffice, the resulting function is non-differentiable. This strategy can also result in sharper increases in stress. This is in contrast to inexact penalty methods using quadratic penalties where the performance depends on the penalty parameter update strategy. 
Inexact penalty methods  require gradually increasing the penalty parameter in each iteration, thus requiring multiple iterations to reach the minimum. The use of gentle penalties means that small changes are made in the coordinates of the nodes resulting in lower stress. Thus, these variations can be used to define objectives with varying emphasis on the stress and intersections components, resulting in different layouts.

Using a multi-objective penalty approach, edge crossing minimization can be incorporated into any optimization-based embedding or graph drawing formulation.  In this paper, we do multi-objective optimization combining edge-cross minimization with the stress function given by the widely-used MDS objective \citep{coxmds}:
\begin{equation}
\label{eq:stressfunc}
stress(X)=\sum_{i<j}w_{ij}(||X_i-X_j||_2-d_{ij})^2
\end{equation}
for $X \in \mathbb{R}^{n \times 2}$. The stress function measures the difference of the Euclidean distance between points in the new reduced two-dimensional space and their corresponding defined m-dimensional proximities. It is therefore a measure of the disagreement between pairwise distances in the high-dimensional space and the new reduced space. Here $X_i \in \mathbb{R}^2$ is the position of the node $i$ in the two-dimensional embedding described by its x and y coordinates. $d_{ij}$ represents the distance between nodes i and j. The normalization constant $w_{ij}=d_{ij}^{- \alpha}$, $\alpha=2 $ is  used. This constant $\alpha$ can be tweaked to alter the emphasis on preserving distances between nearby or faraway nodes as needed. Thus, the solution to (\ref{eq:stressfunc}) is a configuration of points in a two-dimensional Euclidean space such that points in this space represent the original objects,  and the pairwise distances between nodes in this space match the original dissimilarities between the objects.
                                  
However, the stress function in (\ref{eq:stressfunc}) is non-convex. The stress majorization or SMACOF algorithm introduced in \citep{smacof} finds an approximate solution to (\ref{eq:stressfunc}) by iteratively minimizing quadratic approximations to the original stress function. Each approximation, known as the majorization function is simpler to minimize than the original stress function and always takes a value greater than or equal to the original function. SMACOF iteratively minimizes the approximation function $Fstress(X,Z)$ which is a quadratic approximation of the $stress$ function in (\ref{eq:stressfunc}). $Fstress(X,Z)$ is defined such that $stress(X) \leq Fstress(X,Z)$ always holds, and $Fstress(X,Z)$ touches the surface of the $stress(X)$ function at a single point $Z$ known as the \emph{supporting point}.
\begin{equation}
Fstress(X,Z)=\sum_{i<j}w_{ij}d_{ij}^2+Tr(X'L^wX)-2Tr(X'L^{Z}Z)
\end{equation}
where the $n \times n$ weighted Laplacian is defined as follows
  \begin{equation*}
L_{i,j}^{w} = \begin{cases}
-w_{ij} &  \,\, i\neq j \text{,} \\
\sum_{i\neq k} w_{ik} &  \,\,  i = j \end{cases}
\end{equation*}
and
  \begin{equation*}
L_{i,j}^{Z} = \begin{cases}
-w_{ij}d_{ij}inv\left(||Z_i-Z_j||_2\right) &  \,\, i\neq j \text{,} \\
\displaystyle -\sum_{i\neq j} L_{i,j}^{Z} &  \,\,  i = j \end{cases}
\end{equation*}
where $inv(x) = \frac{1}{x}$ when $x \neq 0$ and 0 otherwise, and $Z \in \mathbb{R}^{n \times 2}$ is a constant matrix and is the value of $X$ from the previous iteration. 

The addition of 2-norm edge crossing penalties for $m$ crossings produces the following unconstrained optimization problem:
\begin{multline}
\min_{X, U, r} p(X,Z,U,r)=\min_{X, U, r} Fstress(X,Z)+ \\ \sum_{i=1}^m \frac{\rho _i} {2} [|| (-A^i X U_i+(r_i +1)e)_+||_2 ^2+ ||(B^i X U_i-(r_i -1)e)_+||_2^2] 
  \label{eq:problem}
 \end{multline}
where $\rho \in \mathbb{R}^m$ defines the penalty parameters. Here $A^i \in \mathbb{R}^{2\times n}$ represents a matrix that serves to select the coordinates of the end-points of edge object $\cal{A}$ from $X$ to obtain $A= \begin {bmatrix} a_x & a_y \\c_x & c_y \end{bmatrix}$ as defined in Section \ref{sec:cntnsedgecrossing}. $U \in \mathbb{R}^{m \times 2}$ and $r \in \mathbb{R}^m$. $U_i$  and $r_i$ contains the $u$ and $\gamma$ obtained from (\ref{eq:LP}) that define the separating hyperplane for edge-crossing $i$ .

The iterative penalty method given in Algorithm \ref{alg} progressively increases the penalty parameters on each detected edge crossing until the algorithm converges. \footnote{The algorithm was implemented in Matlab. The initial solution $X^1$ with stress ${s^1}$ is calculated using \emph{stress majorization}. The following parameters are used: $\epsilon=1e-3$, $\tau=1e-6$, $constant = 4$, $\rho_{min} = s^1/constant$, $\rho_{inc}=1.1$ and $\rho_{max}=10^6$.} The algorithm begins by finding the stress-majorization solution $X^1$ and then refining the solution by introducing penalties for crossed edges. At each iteration, the edge crossing detection QP (\ref{eq:LP}) is solved to both detect edge crossings and calculate the hyperplanes used in the edge-crossing penalties.  For each crossed edge, the penalty corresponding to that edge pair is increased at each iteration until a maximum penalty is reached.  The penalty is increased slowly to avoid problems with ill-conditioning.  For this work, we emphasized edge crossing minimization so  $\rho_{max}$ is set high.  But $\rho_{max}$ can be reduced to examine the trade-offs between the embedding and edge-crossing minimization objectives.  Computational efficiency is gained due to the fact that edges that never cross in the course of the algorithm have a penalty of 0. A formal analysis of the computational complexity of this algorithm is left for future work.
\begin{algorithm}
\caption{Crossing Reduction with Stress Majorization}
\label{alg}
\begin{algorithmic}
\STATE Input: Pairwise distances, Edge list
\STATE $C \gets \emptyset$
\STATE $U^* \gets [ ]$
\STATE $r^* \gets [ ]$

\REPEAT
\REPEAT
    \FOR {each edge pair $A^i$ and $B^i, i=1, \ldots, m$} 
	\STATE $(u^{*},\gamma^{*}) \gets argmin f(A^i,B^i, U_i, r_i)$   \{defined in QP (\ref{eq:LP})\}
     \IF{$f(A^i,B^i, u^{*},\gamma^{*})\ge \tau$  and $i  \notin C$}
          \STATE $\rho_i \gets \rho_{min}$ 
           \STATE $C \gets C \cup i$
           \STATE $U_i \gets u^{*}$
           \STATE $r_i \gets \gamma^{*}$
     \ENDIF
   \ENDFOR
   \STATE $Z \gets X^j$
   \STATE $L^Z \gets L^{X^j}$
	\STATE $X^{j+1} \gets argmin(p(X,Z,U^*,r^*))$ \{defined in (\ref{eq:problem})\}
\STATE $j \gets j+1$
	      \STATE $\rho_{i} \gets min(\rho_{inc} \times \rho_i, \rho_{max})$
\UNTIL{$||X^j-X^{j-1}|| \geq \epsilon$ \AND $C \not=\emptyset$ }
\UNTIL {$||X^j-X^{j-1}|| \geq \tau$  \AND $ {\partial \over \partial X} p(X^j,X^{j-1},U^*,r^*) \geq \tau$  \AND $C \not=\emptyset$}
\end{algorithmic}
\end{algorithm}

For each fixed value of penalty parameter $\rho$, the objective (\ref{eq:problem}) is minimized using an algorithm that alternates between a $U-phase$ and an $X-phase$. In the $U-phase$, the soft margin separating hyperplane ($u^i$,$\gamma^i$) for each edge pair $i$ for a fixed $X$ is determined by solving QP(\ref{eq:LP}). For crossings that are present in successive iterations, the $u$ and $\gamma$ determined by solving (\ref{eq:LP}) in the previous iteration are used as a ``warm-start" point for improved efficiency. An inexpensive heuristic is used to reduce the number of edge-pairs checked: no crossing possible if bounding boxes enclosing edges do not intersect. In the $X-phase$, the configuration of nodes $X$ is determined by minimizing (\ref{eq:problem}) with respect to $X$ alone for the fixed values of $U$ and $r$ defining the separating hyperplanes determined in the $U-phase$. The BFGS algorithm as implemented in the Matlab Optimization Toolbox is used to minimize this edge-crossing  penalized stress function. The penalties for crossed edges are driven higher until no edge crossings exist or the problem converges. Note, most edge pairs have penalty parameter $\rho_i=0$ since they never cross. Therefore, such an alternating iterative strategy involves solving relatively simpler problems in each phase for each iteration.
\section{Results and Interpretation}
\label{sec:results}
Two sets of graphs were used to evaluate the performance of CR-SM: 
\begin{itemize}
\item{a real-world application: visualization of \emph{phylogenetic forests} (spoligoforests) of MTBC strains}, and 
\item{randomly generated graphs that have a known planar embedding with high stress. The graphs were designed to challenge the CR-SM algorithm to find an alternate embedding that minimizes crossings while keeping the stress low.}
\end{itemize}
\subsection{Embedding of  Spoligoforests}
To demonstrate the performance of the approach, we return to the motivating application: visualization of spoligoforests created from  DNA fingerprints of MTBC strains \citep{reyes2008,webtools}. Each node of the spoligoforest corresponds to a distinct genotype of MTBC as determined by two types of DNA fingerprints: 43 bit long spoligotypes and 12 loci of MIRU. The genetic distance between nodes is measured by the number of distinct changes in the spoligotypes and MIRU \citep{RulesPaper}. Nodes are colored by lineage, which are assigned to strains by experts on the basis of the strains' biomarkers. The adjacency relationships in the spoligoforest are determined as in \citep{RulesPaper}. The graph is always planar by definition since it consists of multiple trees. Note, not all putative evolutionary relationships between strains can be inferred, resulting in a disconnected graph that has ``orphan'' nodes, i.e. nodes not connected to any other components. This presents a challenge to traditional graph drawing methods like spring embedding techniques that use existence of edges as indicators of proximity. Strains belonging to the same lineage are likely to have small genetic distances between them. There may be some deviations from this expectation which may be of interest from a biological perspective as well. Thus, good spoligoforest visualizations must accurately represent pairwise genetic distances between strains belonging to the same lineage as well as across lineages, and also evolutionary distances between the lineages (subgraphs) themselves. 
In this section, we make comparisons of visualizations generated using CR-SM with six existing crossing minimization and proximity-preserving graph embedding methods that were described in Section \ref{sec:Motivation}.

We examine the visualization of spoligoforests with distance matrices defined using spoligotype and MIRU type (MTBC biomarkers) for four problems. The results are summarized in Table \ref{spoligoforeststats} and the visualizations are shown in Figure \ref{fig:SpolDB4} and the appendix. For all graphs, CR-SM was initialized using the MDS produced by the stress majorization algorithm and run until convergence criteria described in \ref{sec:CR-SM}. In all the spoligoforests, the proposed method, CR-SM, drastically reduces the edge crossings, while making only minor changes in the total stress as compared with other proximity preserving methods MDS, Laplacian Eigenmaps and SNE. While the dimensionality reduction techniques Laplacian Eigenmaps and SNE extended to apply to graph data, can represent proximity relations locally, visualizations generated by these methods have a large number of edge crossings obscuring underlying relationships. Additionally, they do not represent all pairwise distances faithfully. As indicated by the stress values in Table \ref{spoligoforeststats}, the CR-SM visualizations are more informative and accurate than those produced by existing popular approaches for drawing planar graphs e.g. embeddings generated using spring, orthogonal and Twopi that disregard genetic distances available in the heterogeneous data \citep{reyes2008, RulesPaper}. All stress values reported in Table \ref{spoligoforeststats} are scaled such that the stress produced by NEATO (stress majorization implementation) is 1. These results show the efficacy of CR-SM in optimizing multiple objectives pertaining to both stress and crossings.

Note while the original graph is in a fifty-five dimensional space, the data is inherently lower dimensional, thus many embeddings are possible with similar stress. In three of the four graphs, CR-SM, i.e. minimizing the majorization function with edge-crossing penalties, actually produced  graphs with less stress than the those generated by NEATO (stress majorization). This illustrates that edge-crossing penalties may help guide stress majorization to a more desirable local minima  with little or no change in the overall stress. The proposed approach can be used to dynamically remove edge crossings in an existing graph. An animation of the proposed algorithm altering the initial MDS solutions can be viewed at \url{http://www.cs.rpi.edu/~shabba/FinalGD/}.

\begin{table}
\label{tab:spolstats}
\centering
\small\addtolength{\tabcolsep}{-1pt}
\tiny{
\begin{tabular}{|c|c|c||c|c|c||c||c||c|c||c|c||}
\cline{1-12}
\multirow{2}{*}{Graph} & \multirow{1}{*}{Num of} & \multirow{1}{*}{Num of} & \multicolumn{3}{|c||}{CR-SM(MDS+Constraints)} & \multirow{1}{*}{MDS}  &\multirow{1}{*}{Planar} & \multicolumn{2}{|c||}{Laplac. Eigenmap} & \multicolumn{2}{|c||}{SNE}\\ \cline{4-12}
& \multirow{1}{*}{Nodes} & \multirow{1}{*}{Edges} & \tiny{stress} & \tiny{Init. \# cr.} & \tiny{Final \# cr.} &  \tiny{\# cr.} & \multirow{1}{*}{\tiny{stress}} &\tiny{stress} &\tiny{\# cr.} &\tiny{stress} &\tiny{\# cr.} \\ \cline{1-12} \hline
\tiny{LAMs}&68& 	66& 	0.91& 	27& 	0& 	43& 	3.12& 		9.77 & 29	 &4.5462 &13\\ \cline{1-12}
\tiny{\emph{M. africanum}}& 	97& 	89& 	0.99& 	11& 	0& 	11& 	6.32& 		18.81 & 210 & 14.0103&4\\ \cline{1-12}
\tiny{H, X, LAM}& 45& 	29& 	0.90& 	9& 	0& 	1& 	8.64& 		15.83 & 87 & 8.7490&41\\ \cline{1-12} 
\tiny{SpolDB4}& 151& 	138& 	1.06& 	51& 	2& 	51& 	3.32& 	 	  12.46 & 207 & 6.0467&39\\ \cline{1-12} \hline
\end{tabular}
}
\caption{Stress and number of crossings (abbreviated \# cr.) in embeddings generated for four spoligoforests by the proposed approach CR-SM. Comparisons made with (i) MDS embedding (NEATO) (ii) Planar embeddings (Twopi) and embeddings generated by alternate proximity preserving methods (iii) Laplacian Eigenmaps and (iv) SNE.}
\label{spoligoforeststats}
\end{table}

\subsection{Randomly Generated Planar Graphs}
In this section, we demonstrate the performance of the method on a suite of randomly generated planar graphs.
In order to evaluate the performance of the algorithm, we generate random graphs that have at least one known planar embedding. However, the graphs were constructed so that the known planar embedding violates the proximity preservation requirement.  The challenge for the CR-SM algorithm is to find alternate embeddings that preserve proximity relations while still keeping the number of crossings low. 

The graphs were generated as follows: $|V|$ points were generated in $\mathbb{R}^n$ with $n$ varying from 7, 15 to 20 to make generating the embedding progressively more challenging. The Euclidean distance between each pair of points was determined. The points were projected in $\mathbb{R}^2$ and $|E|$ edges were introduced between nodes so that planarity is preserved using a Markov Chain algorithm, as per the method in \citep{randomgraphs}. Since the planar embedding has high stress and is not truly representative of the proximity relations in the data, it is not the most desirable embedding. By relaxing the requirement for 0 crossings, CR-SM can find alternate embeddings with few edge crossings and reduced stress, thus achieving a balance between these often contrary objectives. 
\begin{figure}
  \begin{tabular}{cc}
(a) \;\;\; \includegraphics{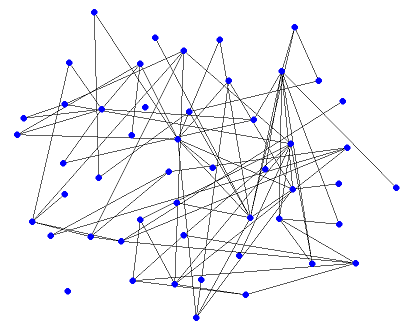}
(b) \;\;\; \includegraphics{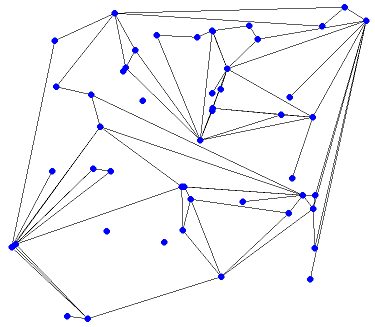}
 \end{tabular}
 \caption{ Embeddings for randomly generated graph in $\mathbb{R}^7$ with 50 nodes and 80 edges using (a) Stress majorization \emph{(stress$=$131.8, 369 crossings)} and (b) CR-SM \emph{(stress$=$272.1, 0 crossings)}. The original  planar embedding had \emph{stress$=$352.5}.}
 \label{fig:randgraphs}
\end{figure}

 \begin{figure}
(a) \includegraphics[scale=0.65]{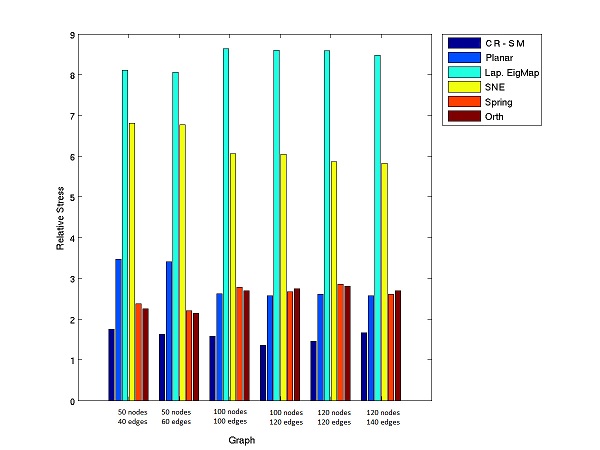}\\
(b) \includegraphics[scale=0.65]{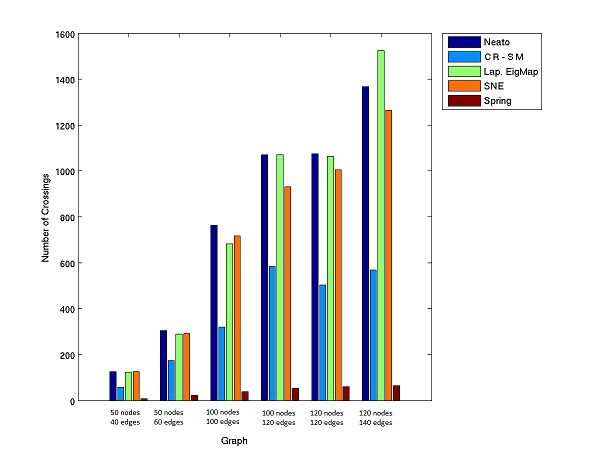}
 \caption{Comparison of (a) stress and (b) number of crossings in embeddings for randomly generated graphs with 50-120 nodes and 40-160 edges generated using CR-SM, NEATO (stress majorization), Laplacian Eigenmaps, SNE, Spring and  Orthogonal embedding. Stress is scaled such that the stress produced by NEATO is 1.}
 \label{fig:randgraphsplots}
\end{figure}

\begin{figure}
\includegraphics[scale=0.5]{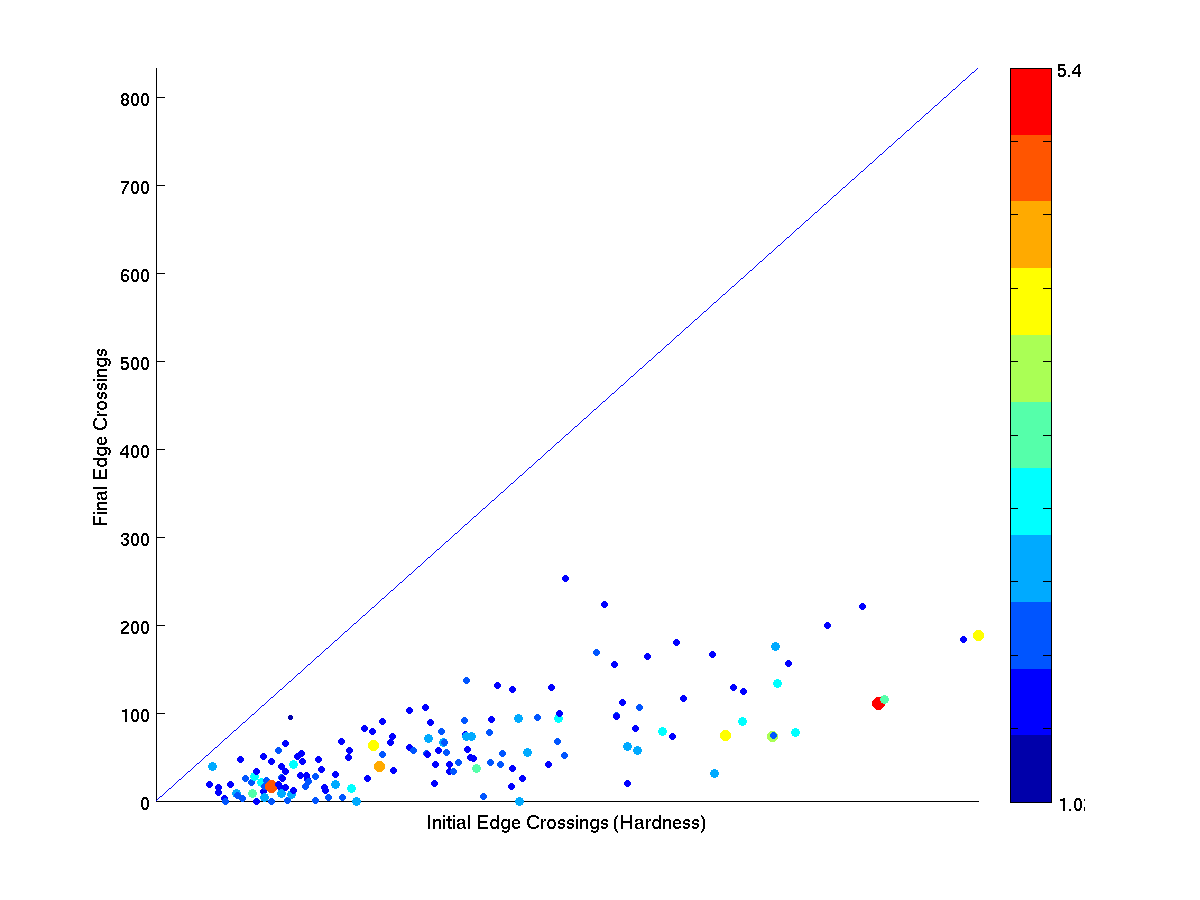}
  \caption{ Plot comparing embeddings generated by CR-SM and stress majorization in terms of stress and number of crossings for 160 randomly generated graphs with 50-80 nodes and 40-160 edges.}
\label{fig:stresscrossrandgraph}
\end{figure}
The plots in Fig. \ref{fig:randgraphsplots} represent the number of crossings and stress of embeddings for a set of 160 randomly generated graphs with 50 to 120 nodes and 40 to 160 edges averaged for all graphs with the same $|V|$ and $|E|$. The number of crossings increases with the graph density. Comparisons are made between CR-SM and the six other algorithms discussed in this paper: MDS, SNE, Laplacian Eigenmaps, Orthogonal Embedding, Spring Embedding and Twopi. Optimizing with respect to the stress alone (MDS), results in embeddings that have many edge crossings. The Laplacian Eigenmap embedding that is optimized with respect to a different proximity preservation objective aimed at preserving local structure also has a large number of crossings. Force-directed placement methods e.g. spring embedding \citep{fdpfruchterman} and planar grid embedding techniques like ORTH \citep{orth} have a low number of edge crossings but can have high stress. Alternate embeddings are generated by CR-SM that keep the number of edge crossings low while preserving proximity relations, thus clearly representing the underlying graph structure i.e. adjacency and connectivity information. Figure \ref{fig:stresscrossrandgraph} shows a plot of final edge crossings vs initial edge crossings in embeddings generated by CR-SM for these 160 randomly generated graphs. The size and color of the nodes represent the ratio of final stress of the embeddings found by CR-SM to that of the stress majorization solution found by NEATO. As indicated by the vast majority of blue dots in the lower half of the plot, CR-SM can produce embeddings with a significant reduction in the number of crossings with small increase in stress as compared to the graph embeddings produced using only the MDS objective. 

\section{Discussion}
In this work we introduce a fundamentally new paradigm for elimination of edge crossings in graph embeddings. We developed a novel approach to simultaneously optimizing the aesthetic criteria for no edge crossings and preservation of proximity relations in heterogeneous graph data.  This work demonstrates how edge-crossing constraints can be formulated as a system of nonconvex constraints.  Edges do not cross if and only if they can be strictly separated by a hyperplane. If the edges cross, then the hyperplane defines the desired half-spaces that the edges should lie within. The edge-crossing constraints can be transformed into a continuous edge-crossing penalty function in either 1-norm or least-squares form.  We developed the Crossing Reduction with Stress Majorization (CR-SM) algorithm that couples the stress majorization algorithm with a penalty method for edge crossing minimization. Computational results demonstrate that this approach is quite practical and tractable. Continuous optimization methods can be used to effectively find local solutions, a very desirable outcome since drawing graphs with a minimum number of edge-crossings is NP-hard.  Successful results were illustrated on problems in the epidemiology of tuberculosis involving visualizing phylogenetic forests that were not adequately addressed using existing planar graph drawing approaches since they did not preserve proximity relations and gave especially undesirable results on disconnected graphs.  Dimensionality reduction techniques such as Laplacian Eigenmaps and SNE applied to this problem were successful in depicting proximities amongst immediate neighbors, but they failed to represent all pairwise distances. Moreover, they have many crossings. Whereas stress majorization, a popular means of solving MDS, achieves low stress, it also has a large number of crossings. Interestingly, CR-SM actually found embeddings with less stress and reduced number of crossings as compared with the stress majorization solution. This may be caused by the fact that  the MTBC data is from a fifty-five dimensional space and the stress function is highly nonconvex with  many possible locally optimal embeddings existing with similar stress, thus the edge crossing constraints may help guide the algorithm to a more desirable local solution both from a stress and aesthetic point of view. Results on high dimensional random graphs with planar embeddings show that the method can find much more desirable solutions from a visualization point of view with only relatively small changes in stress.   

This work opens up many avenues for future research at the intersection of machine learning and data visualization. Here we focused on elimination of edge crossings and stress minimization (MDS).  The general multiobjective approach for minimizing overlaps between graph components is applicable to any optimization-based dimensionality reduction, graph drawing or embedding methods \citep{van2007dimensionality, qpdwyer} used for data visualization in both  supervised and unsupervised learning.  While the method described was motivated by the need to minimize edge crossings and simultaneously preserve pairwise distances in heterogeneous graph data as defined by the MDS objective, it can be used to eliminate edge crossings with any embedding objective. The theorems and algorithms are directly applicable to the intersection of any graph components that are convex polyhedrons.
Thus, the method can also be used to eliminate node-node overlaps and node-edge crossings. Our preliminary work was limited to planar graphs, but the penalty approach can be used to reduce crossings in nonplanar graphs as well. Since the edge-crossing constraints are very closely related to linear SVM, all the different classification and regularization loss functions in SVM could be used to produce crossing-penalty functions with different aesthetic effects and algorithmic ramifications.  For example, maximum margin separation can enforce maximum spacing between graph components. This work used the Matlab function ``fminunc" as its primary workhorse -- which inherently limits the problem size. In reality, there is a great potential for making highly scalable special purpose algorithms for edge-crossing-constrained graph embeddings. The state-of-the-art linear SVM algorithms which are massively scalable can potentially be adapted to this problem as well. We leave these promising research directions as future work.


\acks{We would like to thank Prof. B\"ulent Yener of Rensselaer Polytechnic Institute for his valuable suggestions. We would like to acknowledge support for this project from NIH (R01LM009731).}


\newpage
\appendix
\section*{Appendix A}
We present embeddings of spoligoforests of LAM, West African and H-X-LAM lineages in this appendix. For each spoligoforest we show the visualization generated by the proposed approach CR-SM, proximity preserving embeddings generated using MDS, SNE and Laplacian Eigenmaps and planar embeddings generated using Twopi, Spring and Orthogonal embedding algorithms. The planar embeddings are visually appealing, but genetic distances between strains are not faithfully reflected. Orphan nodes are assigned arbitrary positions and relative placement of the lineages does not reflect the domain knowledge about putative evolutionary distances between lineages. Laplacian Eigenmap produce highly dense clusters of nodes that completely occlude visualization of edges. Both SNE and Laplacian Eigenmaps do not represent \emph{all} pairwise distances. Moreover, they have a large number of crossings. Embeddings generated using NEATO that optimize the MDS objective preserve proximity relations but have many edge crossings. The proposed approach CR-SM eliminates all edge crossings with little change in the overall stress. 
\begin{figure}
\includegraphics[scale=0.5]{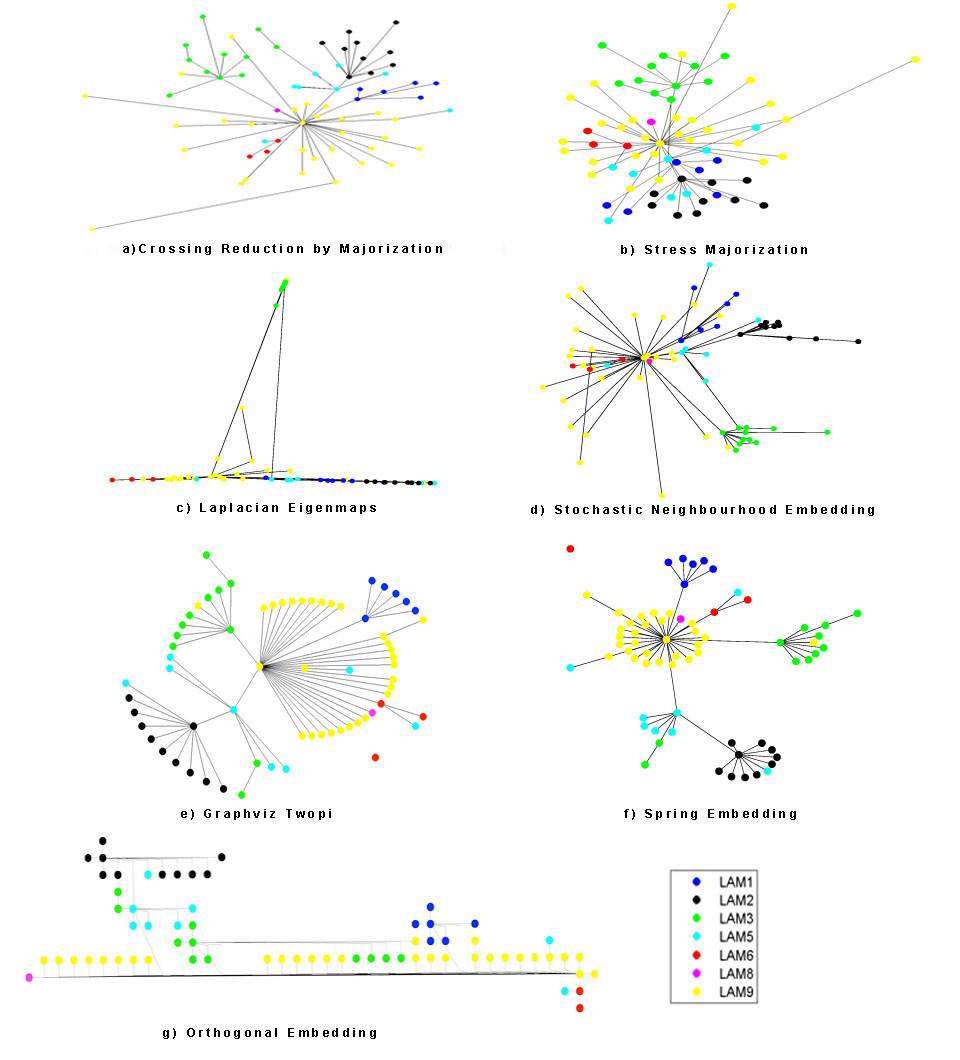}
\caption{Embeddings of spoligoforests of LAM sublineages by 7 algorithms 
(a) CR-SM
(b) MDS
(c) Laplacian Eigenmaps
(d) Stochastic Neighborhood Embedding (SNE)
(e) Graphviz Twopi
(f) Spring Embedding
(g) Orthogonal Embedding.}
\end{figure}

\begin{figure}
\includegraphics[scale=0.5]{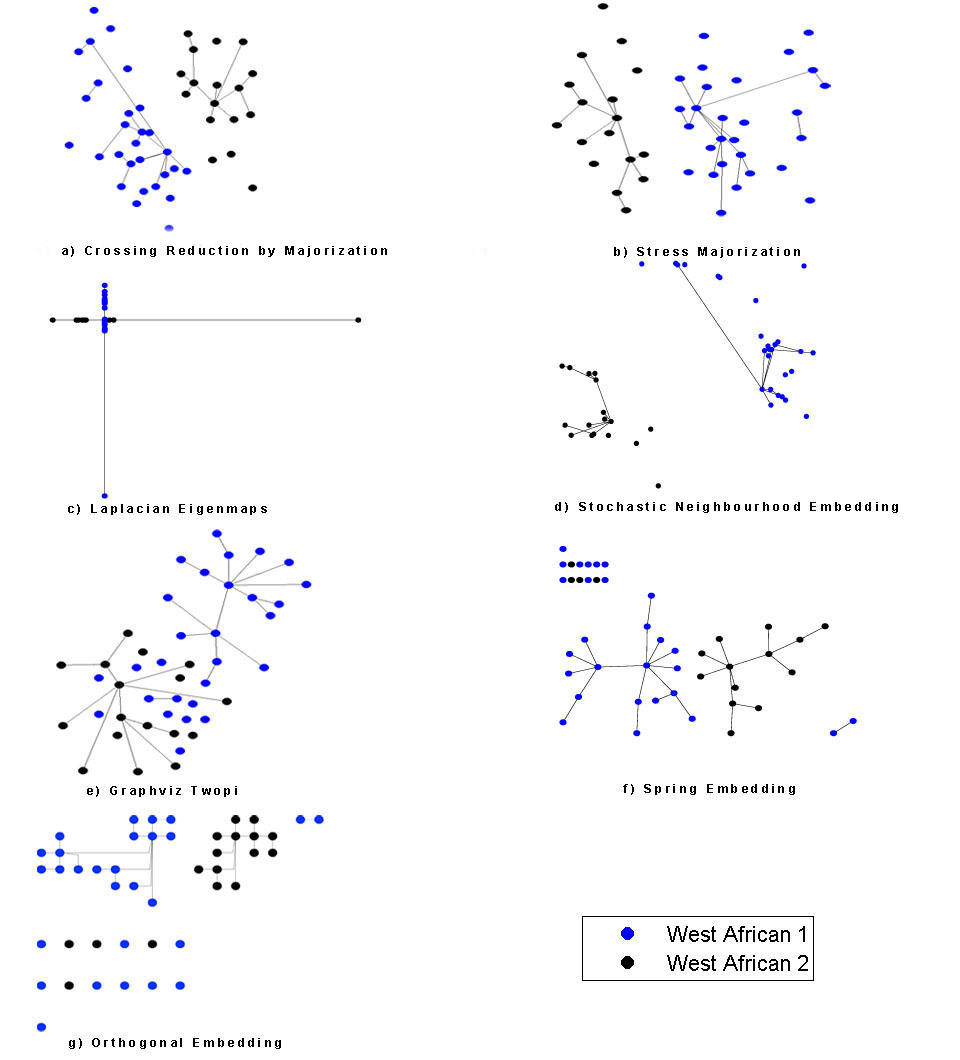}
\caption{Embeddings of spoligoforests of \emph{M. afcricanum} sublineages by 7 algorithms 
(a) CR-SM
(b) MDS
(c) Laplacian Eigenmaps
(d) Stochastic Neighborhood Embedding (SNE)
(e) Graphviz Twopi
(f) Spring Embedding
(g) Orthogonal Embedding.}

\end{figure}
\begin{figure}
\includegraphics[scale=0.5]{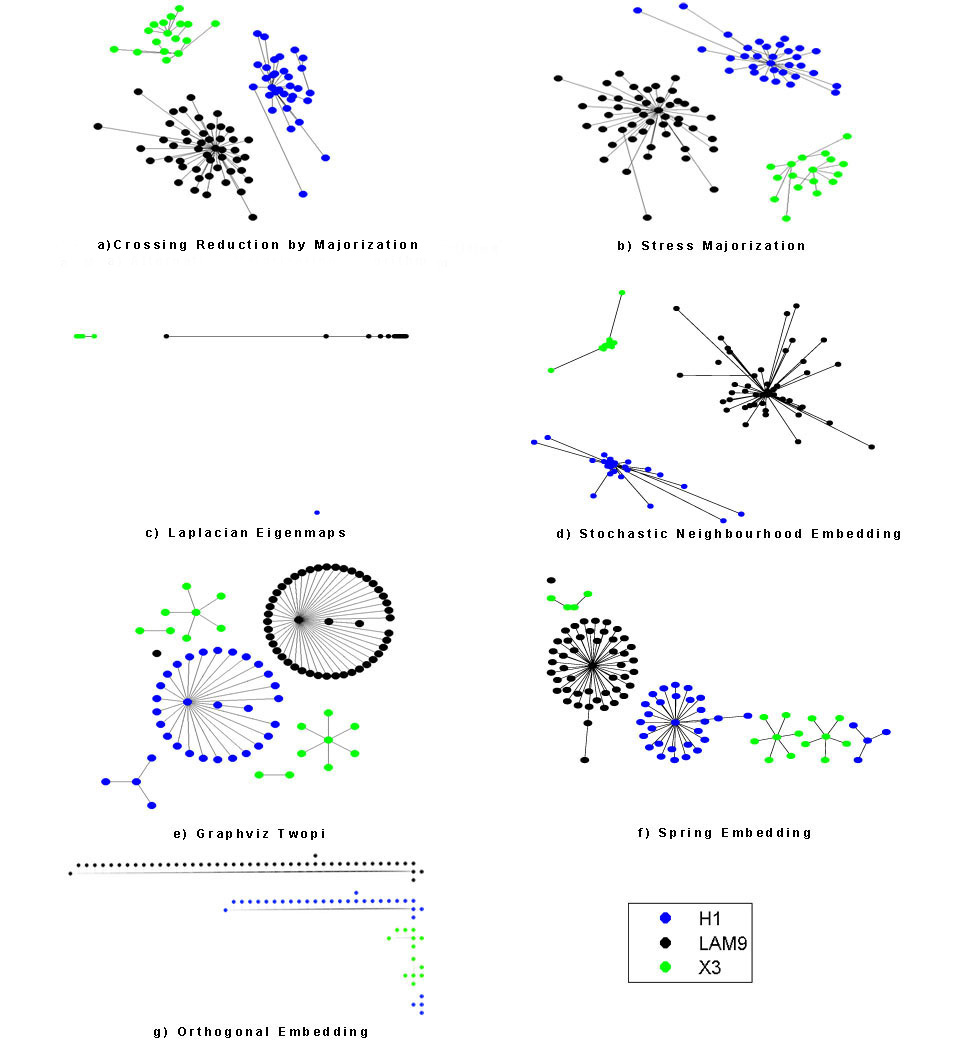}
\caption{Embeddings of spoligoforests of H-X-LAM sublineages by 7 algorithms 
(a) CR-SM
(b) MDS
(c) Laplacian Eigenmaps
(d) Stochastic Neighborhood Embedding (SNE)
(e) Graphviz Twopi
(f) Spring Embedding
(g) Orthogonal Embedding.}
\end{figure}

\newpage
\bibliography{candidacybib2}

\end{document}